\documentclass[12pt]{amsart}

\usepackage{color}

\usepackage{amssymb, amsmath}
\usepackage{enumerate}
\usepackage{graphicx,amssymb,amsmath}
\usepackage{amsthm}
\usepackage{multirow}
\numberwithin{equation}{section}
\newtheorem{lemm}{Lemma}[section]
\newtheorem{thrm}{Theorem}[section]

\newtheorem{rmk}{Remark}[section]
\newtheorem{cor}{Corollary}[section]

\begin{document}
\title[Bloch Band-based Gaussian beam superposition]{Error Estimates of the Bloch Band-Based Gaussian Beam Superposition for the Schr\"odinger Equation}
\author{Hailiang Liu and Maksym Pryporov  \\
\\
(Dedicated to our friend James Ralston)}
\address{Department of Mathematics, Iowa State University, Ames, Iowa 50010}
\email{hliu@iastate.edu; pryporov@iastate.edu}
\maketitle
\begin{abstract}
This work is concerned with asymptotic approximations of the semi-classical Schr\"odinger equation in periodic media using Gaussian beams.  For the underlying equation, subject to a highly oscillatory initial data, a hybrid of the Gaussian beam approximation and homogenization leads to the Bloch eigenvalue problem and associated evolution equations for Gaussian beam components in each Bloch band.   We formulate a superposition of  Bloch-band based Gaussian beams to generate high frequency  approximate solutions to the original wave field.   For initial data of a sum of finite number of band eigen-functions, we prove that the first-order Gaussian beam superposition converges to the original wave field at a rate of $\epsilon^{1/2}$, with $\epsilon$ the semiclassically scaled constant,  as long as the initial data for Gaussian beam components in each band are prepared with same order of error or smaller.  For a natural choice of initial approximation, a rate of $\epsilon^{1/2}$ of initial  error is verified.
\end{abstract}


\section{Introduction}
We consider the semiclassically scaled Schr\"odinger equation with a periodic potential:
\begin{equation}\label{1.1}
i\varepsilon \partial_t \Psi=-\frac{\varepsilon^2}{2}\Delta\Psi+V\left(\frac{x}{\varepsilon}\right)\Psi+V_e(x)\Psi,\quad x\in{\mathbb{R}^d,}\quad t>0,
\end{equation}
subject to the two-scale initial condition:
\begin{equation}\label{1.2}
\Psi(0,x)=g\left(x,\frac{x}{\varepsilon}\right)e^{iS_0(x)/{\varepsilon}},\quad x\in{\mathbb{R}^d,}
\end{equation}
where $\Psi(t,x)$ is a complex wave function, $\varepsilon$ is the re-scaled Planck constant, $V_e(x)$-- smooth external potential, $S_0(x)$-- real-valued smooth function, $\displaystyle g(x,y)=g(x,y+2\pi)$-- smooth function, compactly supported in $x,$ i.e., $g(x,y)=0,\; x\not\in K_0,\; K_0$-- is a bounded set. $V(y)$ is periodic with respect to the crystal lattice $\Gamma=(2\pi\mathbb{Z})^d,$ it models the electronic potential generated by the lattice of atoms in the crystal \cite{DGR06}.

A typical application arises in solid state physics where (\ref{1.1}) describes the quantum dynamics of Bloch electrons moving in a crystalline lattice (generated by the ionic cores)  \cite{Sl49}.  The asymptotics of (\ref{1.1}) as $ \varepsilon \to 0+$   is a well-studied two-scale problem in the physics and mathematics literature \cite{Bu87, GMS91, Ho93, SN99, HST01, DGR02, PST03, AP05, DGR06}. On the other hand, the  computational challenge because of the small parameter $\varepsilon$ has prompted a search for asymptotic model based numerical  methods, see e.g.,  \cite{LW09, JWYH10}.


The main feature of this type of problems is the ``band structure" of solutions. For suitable initial data, the solution depends on the semi-classical Hamiltonian operator
\begin{equation}\label{1.3}
H(k,y)=\frac{1}{2}(-i\nabla_y+k)^2+V(y), \quad {y\in \Gamma},
\end{equation}
and the solution of the eigenvalue problem:
\begin{equation}\label{1.4}
\begin{cases} H(k,y)z(k,y)=E(k)z(k,y),\\
z(k,y)=z(k,y+2\pi),\\
\end{cases}
\end{equation}
where $k\in[-1/2,1/2]^d$-- called Brillouin zone, see \cite{Zi72}. The behavior of the eigen-pairs for general $k$ can be characterized by that for $k$ in this zone through a periodic extension.

According to the theory of Bloch waves \cite{Wi78}, the self-adjoint semi-bounded operator
$H(k,y)$ with a compact resolvent has a complete set of orthonormal eigenfunctions $z_n(k,y)$ in $L^2$, with $e^{iky}z_n(k, y)$  called Bloch functions. The correspondent eigenvalues $E_n(k)$ are called band functions. Standard perturbation theory  \cite{Ka80}  shows that $E_n(k)$ is a continuous function of $k$ and real analytic in a neighborhood of any $k$ such that
\begin{equation}\label{1.5}
 E_{n-1}(k)< E_n(k)<E_{n+1}(k).
\end{equation}
{The proof has been given first in \cite{LH71} and \cite{Cl64a, Cl64b} for $d=1$ and in \cite{Ne1} for $d=3$.} We assume that (\ref{1.5}) is satisfied, i.e.,  all band functions are strictly separated, $\forall n, \; k.$  Under this assumption one can choose $z_n(k,y)$ associated to $E_n(k)$ to be real analytic functions of $k$ \cite{DGR06}.  {This allows for a unique analytic extension of both $z_n(k, y)$ and $E_n(k)$ so that they can be evaluated for some complex $k$, say $k=\partial_ x \Phi$, where $\Phi$ is the Gaussian beam phase. }

A classical approach to solve this problem asymptotically is by the Bloch band decomposition based WKB method \cite{BLP78, GRT88, Sp96}, which leads to Hamilton-Jacobi and transport equations valid up to caustics.  The Bloch-band based level set method was introduced in \cite{LW09} to compute crossing rays and position density beyond caustics.  However, at caustics, neither method gives correct prediction for the amplitude.  A closely related alternative to the WKB method is the construction of approximations based on Gaussian beams. Gaussian beams are asymptotic solutions concentrated on classical trajectories for the Hamiltonian, and  they remain valid beyond ``caustics".   The existence of Gaussian beam solutions has been known since sometime in the 1960's, first in connection with lasers, see Babi\v{c} and Buldyrev \cite{VB56, BB72}. Later, they were used to obtain results on the propagation of singularities in solutions of PDEs \cite{LH71, JR82}.   The idea of using sums of Gaussian beams to represent more general high frequency solutions  was first introduced by Babi\v{c} and Pankratova in \cite{BP73} and was later proposed as a method for wave propagation by Popov in \cite{MP82}.  At present there is considerable interest in using superpositions of beams to resolve high frequency waves near caustics. This goes back to the geophysical applications in \cite{CPP:1982, Hill01}.  Recent work in this direction includes \cite{JR05,TQR07, JWY08, LQ08, Tan08, MR09, JWYH10}.

The accuracy  of the Gaussian beam superposition to approximate the original wave field is important, but determining the error of the Gaussian beam superposition 
is highly non-trivial, see the conclusion section of the review article by Babi\v{c} and Popov \cite{BP89}.  In the past few years, some significant progress on estimates of the error has been made.  One of the first results was obtained by Tanushev for the initial error in 2008 \cite{Tan08}.  Liu and Ralston \cite{LR09, LR10} gave rigorous convergence rates in terms of the small wave length for  both the acoustic wave equation in the scaled energy norm and the Schr\"{o}dinger equation in the $L^2$ norm. At about the same time, error estimates for phase space beam superposition were obtained by  Bougacha, Akian and Alexandre in \cite{BAA09} for the acoustic wave equation.   Building upon these advances,  Liu, Runborg and Tanushev further obtained sharp error estimates for a class of high-order, strictly hyperbolic  partial differential equations \cite{LRT11}.


Other methods that also yield an asymptotic description for time-scales of order $O(1)$ (i.e. beyond caustics) have been developed such as those based on Wigner measures \cite{GMMP97}.  The dynamics of the Wigner function corresponding to the Schr\"{o}dinger wave function can be semiclassically approximated to an error of order $O(\epsilon)$, see \cite{TP04} and references to previous works therein.  More recently, so-called space-adiabatic perturbation theory has been used to derive an effective Hamiltonian, governing the dynamics of particles in periodic potentials under the additional influence of slowly varying perturbations \cite{HST01, PST03}.  The semi-classical asymptotics of this effective model is then obtained in a second step, invoking an Egorov-type theorem. 
{Another analogous approach is the propagation of the so called semiclassical wave packets, developed by Hagedorn et al.  \cite{Ha80}. 
A recent rigorous analysis is given by Carles and Sparber \cite{CS12} in the context of  the Schr\"{o}dinger equation with periodic potentials. There the authors prove that using semiclassical wave packets within each Bloch band, an approximation result up to errors of order $O(\epsilon^{1/2})$ can be achieved, but for
times up to the Ehrenfest time-scale $T~\ln(\frac{1}{\varepsilon})$.  }

In this paper,  we develop a convergence theory for the Gaussian beam superposition as a valid approximate solution of problem  (\ref{1.1})-(\ref{1.2}).  { The novel contribution of the present work lies in the accuracy justification for an explicit construction -- the Gaussian beam superposition. Indeed,  Gaussian beam methods  are widely used in numerical simulations of high frequency waves fields.}

The Gaussian beam construction is based on Gaussian beams in each Bloch band, and carried out by using the two scale expansion approach, essentially following DiMassi et al. \cite{DGR06} for adiabatic perturbations. 
The accuracy study in \cite{DGR06} was  only on how well each Gaussian beam asymptotically satisfies the PDE.   In order to handle more general initial data in this paper,  we   
(i) present the approximation solution through beam superpositions over Bloch bands and initial points from which beams are issued;  and
(ii) estimate the error between the exact wave field and the asymptotic ones.   Numerical results using this type of superpositions were presented in \cite{JWYH10}.

Our focus in this work is mainly on  (ii).  {We use the notation:  $f\in C_b^m(\mathbb{R}^d)$ means that $f$ is $m$ times differentiable function,  and all derivatives up to $m$-th order included are bounded functions in $\mathbb{R}^d$. $f\in L_x^2$ means that f belongs to $L^2(\mathbb{R}^d)$ in $x$ variable}.  The main result can be stated as follows.
\begin{thrm}   Suppose that $S_0\in C^3_b(\mathbb{R}^d),  V_e\in C^{d+4}_b(\mathbb{R}^d)$, both $V(y)$ and $g(x, y)$ are periodic in $y$ with respect to the crystal lattice $\Gamma=(2\pi\mathbb{Z})^d$,  also {$V\in C^2(\Gamma)$} and $g(x, y)$  has compact support in $x$. We also assume  $g$ has the following expression
$$
\displaystyle g(x, y)=\sum_{n=1}^Na_n(x)z_n(\nabla_xS_0(x),y),
$$
where $z_n(k, y)$ are eigen-functions of (\ref{1.4}) with eigenvalues $E_n(k)$ satisfying (\ref{1.5}).  Let $\Psi(t, x)$ be the solution to (\ref{1.1})-(\ref{1.2}), and
$$
\displaystyle\Psi^\epsilon(t, x)=\tilde{\Psi}^\varepsilon\left(t,x,\frac{x}{\varepsilon}\right)
$$
be the Gaussian beam superposition defined by (\ref{3.24}) for $0<t\leq T$,  then
$$
\|\Psi-\Psi^\epsilon\|_{L^2_x} \leq C \varepsilon^{1/2},
$$
where $C$ may depend on $T$, $N$ and data given, but independent of $\varepsilon$.
\end{thrm}
\begin{rmk}
{The regularity requirement of $V$ is sufficient for validating the Gaussian beam approximation, but excludes the Coulomb-like singularity which is typical of the mean field electrostatic potential in real solids. It would be interesting to investigate how such an assumption could be relaxed. }
\end{rmk}
 We prove this result in several steps. We first reformulate the problem using the two scale expansion method \cite{BLP78, DGR06}, in which both $x$ and $y=\frac{x}{\varepsilon}$ are regarded as two independent variables. The well-posedness estimate for this reformulated problem tells that the total error is bounded by the sum of initial and evolution error.  {For initial error, we use some techniques similar to those developed by Tanushev \cite{Tan08}, except that here we have to deal with the band structure. The band structure induces additional technical difficulties, which we solve in several steps. As for evolution error part, we rely on the non-squeezing argument proved in \cite{LRT11}, which is the key technique for the proof. After we obtain estimate in $L^2_{x,y}$ we convert to $L^2_x.$
 }

This paper has the following structure: in section 2 we use the two scale method to reformulate our problem and state the corresponding results; in the end of this section we prove Theorem $1.1$ for the original problem.   In section 3 we review Gaussian beam constructions and formulate our Gaussian beam superposition. Justifications of  main results are presented in section 4 and section 5. In section 6 we discuss possible extensions of our results and some remaining challenges.

\section{Set-up and Main Results}

In order to construct an asymptotic solution of (\ref{1.1}) we use the two-scale method as in \cite{BLP78, DGR06}. We regard $x$ and $\displaystyle y=\frac{x}{\varepsilon}$ as independent variables and introduce a new function $$\tilde{\Psi}(t,x,y)\equiv\Psi(t,x),$$ equation (\ref{1.1}) can be rewritten in the form:
\begin{equation}\label{2.1}\displaystyle\begin{cases}
i\varepsilon \partial_t \tilde{\Psi}=-\frac{1}{2}(\varepsilon\nabla_x+\nabla_y)^2\tilde{\Psi}+V(y)\tilde{\Psi}+V_e(x)\tilde{\Psi},\\
\tilde{\Psi}(0,x,y)=g(x,y)e^{iS_0(x)/\varepsilon},\quad x\in{\mathbb{R}^d,}\quad y\in[0,2\pi]^d.
\end{cases}\end{equation}
We assume that the initial amplitude $g(x,y)$ can be decomposed into $N$ bands,
\begin{equation}\label{gg}
g(x,y)=\sum_{n=1}^Na_n(x)z_n(\nabla_x S_0,y),
\end{equation}
where $a_n$ is determined by
\begin{equation}\label{2.3}
a_n(x)=\int_{[0,2\pi]^d}g(x,y)\overline{z_n(\nabla_x S_0,y)}dy,
\end{equation}
 and $\displaystyle\big\{z_n(\partial_xS_0,y)\big\}_{n=1}^\infty$ are eigenfunctions of the self-adjoint second order differential operator $H(k,y)$ defined by (\ref{1.3}). $\displaystyle\big\{z_n(\partial_xS_0,y)\big\}_{n=1}^\infty$ form an orthonormal basis in $L^2(0,2\pi).$

For each energy band, the Gaussian beam ansatz was constructed in \cite{DGR06}, which we will review in section $3:$
\begin{equation}\label{2.4}
\tilde{\Psi}_{GB}^n(t,x,y;x_0)={A}^n(t,x,y;x_0)e^{i\Phi_n(t,x;x_0)/\varepsilon},
\end{equation}
where $\Phi_n$ and $ {A}^n$ are Gaussian beam phases and amplitudes,  respectively, $n=1,\dots N$. The Gaussian beam phase is defined as:
\begin{equation}
\Phi_n(t,x;x_0)=S_n(t;x_0)+p_n(t;x_0)(x-\tilde{x}_n(t;x_0))+\frac{1}{2}(x-\tilde{x}_n(t;x_0))^\top M_n(t;x_0)(x-\tilde{x}_n(t;x_0)),
\end{equation}
where $\tilde{x}_n, p_n, S_n$ and $M_n,$ as well as the amplitude $a_n$ satisfy corresponding evolution equations  (see section 3 for details).
Using the fact that the Schr\"odinger equation is linear, we sum the Gaussian beam ansatz for each band to obtain the approximate solution along the ray:
\begin{equation}
\tilde{\Psi}_{GB}(t,x,y;x_0)=\sum_{n=1}^N\tilde{\Psi}_{GB}^n(t,x,y;x_0).
\end{equation}
Using $\tilde{\Psi}_{GB}(t,x,y;x_0)$ as a building block of the approximate solution,  we have the following superposition of Gaussian beams:
\begin{equation}
\tilde{\Psi}^{\varepsilon}(t,x,y)=\frac{1}{(2\pi\varepsilon)^{\frac{d}{2}}}\int_{K_0}\tilde{\Psi}_{GB}(t,x,y;x_0)dx_0,
\end{equation}
where $\displaystyle\frac{1}{(2\pi\varepsilon)^{\frac{d}{2}}}$ is a normalizing constant which is needed for matching the initial data of problem (\ref{2.1}).   The initial data is approximated by:
\begin{equation}\label{2.8}
\tilde{\Psi}^{\varepsilon}(0,x,y)=\frac{1}{(2\pi\varepsilon)^{\frac{d}{2}}}\int_{K_0}\sum_{n=1}^NA^n(0,x,y;x_0)e^{i\Phi^0(x;x_0)/\varepsilon}dx_0,
\end{equation}
where ${A}^n(0,x,y;x_0)$ is the initial data for the amplitude, and  $\Phi^0$ is the initial Gaussian beam phase for all bands,  chosen as follows:
\begin{equation}\label{2.9}
\Phi^0(x;x_0)=S_0(x_0)+\nabla_xS_0(x_0)\cdot(x-x_0)+\frac{1}{2}(x-x_0)^\top\cdot(\nabla_x^2 S_0(x_0)+iI)(x-x_0).
\end{equation}
We address the two-scale problem, with $\displaystyle y=\frac{x}{\varepsilon}$ considered to be independent variables, and then convert to the original problem.
The norm $L^2_{x,y}$ is defined as follows:
\begin{equation}
\|u\|^2_{L^2_{x,y}}=\int_{[0,2\pi]^d}\int_{\mathbb{R}^d}|u(x,y)|^2dxdy.
\end{equation}
 We obtain two major results formulated in the following theorems:
\begin{thrm}\label{th2.1}[Initial error estimate]
Let $K_0\subset\mathbb{R}^d$ be a bounded measurable set, $g(x,y)\in H^1(K_0\times[0,2\pi]^d),\; S_0(x)\!\in C_b^3(\mathbb{R}^d)$. Then the initial error made by the Gaussian beam
superposition (\ref{2.8}) is as follows:
$$\|\tilde{\Psi}(0, x, y)-\tilde{\Psi}^\varepsilon(0, x, y)\|_{L^2_{x,y}}\leq C\varepsilon^{1/2},$$
where constant $C$ depends only on the initial amplitude $g(x,y)$ and the initial phase $S_0(x).$
\end{thrm}
The proof is split in two parts, see Lemma \ref{4.1} and Lemma \ref{4.2}.

In order to measure the evolution error, we define $P$ the two-scale Schr\"odinger operator,
\begin{equation} \label{2.11}
P(\tilde{\Psi})=i\varepsilon\partial_t\tilde{\Psi}+\frac{1}{2}(\varepsilon\nabla_x+\nabla_y)^2\tilde{\Psi}-V(y)\tilde{\Psi}-V_e(x)\tilde{\Psi}.
\end{equation}

 \begin{thrm} \label{Th2.2}[Evolution error estimate]
Let $K_0$ be a bounded set, condition (\ref{1.5}) is satisfied, the external potential $V_e(x)\in C^{d+4}_b(\mathbb{R}^d)$.
Then the evolution error is
$${\rm sup}_{0\leq t\leq T} \|P(\tilde{\Psi}^\varepsilon(t,\cdot))\|_{L^2_{x,y}}\leq C\varepsilon^{3/2},$$
where constant $C$ depends on the measure of set $K_0,$ finite time $T$, the number of bands $N$, 
and external potential $V_e.$
\end{thrm}
The proof of this theorem is done in several steps,  one step requires  a phase estimate which uses essentially the ``Non-squeezing" result obtained by Liu et al. \cite{LRT11}.

Finally  we recall the well-posedness estimate for the two-scale Schr\"odinger equation  (\ref{2.1}).
\begin{lemm}
The $L^2$--norm of the difference between the exact solution $\tilde{\Psi}$ and an approximate solution $\tilde{\Psi}^\varepsilon$ of the problem (\ref{2.1}) is bounded above by the following estimate:
\begin{equation}
\|\tilde{\Psi}(t, x, y)-\tilde{\Psi}^\varepsilon(t, x, y)\|_{L^2_{x,y}}\leq\|\tilde{\Psi}(0, x, y)-\tilde{\Psi}^\varepsilon(0, x, y)\|_{L^2_{x,y}}+\frac{1}{\varepsilon}\int_0^T\|P(\tilde{\Psi}^\varepsilon)\|_{L^2_{x,y}}dt,\quad 0<t\leq T,
\end{equation}
where $T$ is a finite time,  $\tilde{\Psi}(0,\cdot), \tilde{\Psi}^\varepsilon(0,\cdot)$ are initial values of the exact and approximate solution respectively.
\end{lemm}
This result when combined with both initial error and  evolution error gives the following.
\begin{cor}\label{cor2.1}
The total error made by the first order Gaussian beam superposition method is  of order $\varepsilon^{1/2}$ in the following sense
$$
\|\tilde{\Psi}-\tilde{\Psi}^\varepsilon\|_{L^2_{x,y}}\leq C\varepsilon^{1/2}.
$$
\end{cor}

In order to convert the two-scale result stated in Corollary \ref{cor2.1}  to  the original problem, we prepare the following lemma.
\begin{lemm} \label{lem2.1} Assume that  $f(x,y)\in L^2(\mathbb{R}^d, [-\pi, \pi]^d)$ and $f$ is $2\pi$ periodic in $y.$
Then for sufficiently small $\varepsilon$,
\begin{equation}\label{fxy}
\left\|f\left(x, \frac{x}{\varepsilon}\right)\right\|_{L^2_x} \leq \frac{1}{\pi^{\frac{d}{2}}} \|f(x, y)\|_{L^2_{x,y}}.
\end{equation}
\end{lemm}
\begin{proof}
Denote $\displaystyle Y_\varepsilon^k=[2\pi k\varepsilon,2\pi(k+1)\varepsilon]^d$ and let $ I_\varepsilon=\{k\in\mathbb{Z}^d,\quad Y_\varepsilon^k\cap [-R, R]^d\neq\emptyset\}$ for any fixed $R>0$.
Then,
\begin{equation*}
\int_{|x|\leq R}f^2\left(x,\frac{x}{\varepsilon}\right)dx\leq\sum_{k\in I_\varepsilon}\int_{Y_\varepsilon^k}f^2\left(x,\frac{x}{\varepsilon}\right)dx.
\end{equation*}
Here $|x|$ denotes $l^\infty$-- norm of the vector $x,$ hence $|x|\leq R$ corresponds to a $d$-dimensional cube.
Introducing a change of variable $y=\frac{x}{\varepsilon}$ and taking advantage of the periodicity in $y,$ one can rewrite the right hand side of the above expression in the shifted cell form:
\begin{equation*}
\int_{|x|\leq R}f^2\left(x,\frac{x}{\varepsilon} \right)dx\leq\sum_{k\in I_\varepsilon}\varepsilon^d\int_{|y|\leq\pi}f^2(\varepsilon(y+2\pi k),y)dy.
\end{equation*}
For fixed $y$ the right hand side corresponds to the Riemann sum of the function $$\displaystyle g^2(x)=\int_{|y|\leq\pi} f^2(x+ \varepsilon y,y)dy$$ sampled at $x_k=2\pi k  \varepsilon$. Note that
the step size in all direction $\Delta x_k= (2\pi  \varepsilon)^d$, hence
\begin{align*}
\sum_{k\in I_\varepsilon}\varepsilon^d \int_{|y|\leq\pi}f^2(\varepsilon(y+2\pi k),y)dy & = \frac{1}{(2\pi)^d} \sum_{k\in I_\varepsilon} g^2(x_k)\Delta x_k \\
&  \rightarrow\frac{1}{(2\pi)^d}\int_{|x|\leq R} \int_{|y|\leq\pi}f^2(x,y)dydx\quad\mbox{as}\quad\varepsilon\rightarrow0,
\end{align*}
with the first order of convergence. Therefore,
\begin{equation*}
\int_{|x|\leq R}f^2\left(x,\frac{x}{\varepsilon} \right)dx\leq\frac{1}{(2\pi)^d}\int_{|x|\leq R} \int_{|y|\leq\pi}f^2(x,y)dydx+C\varepsilon.
\end{equation*}
Taking $\displaystyle C=(2\pi)^{-d} \|f\|^2_{L^2_{x, y}}$ and $\varepsilon<1$, then the right hand side is bounded above by $$\displaystyle\frac{1}{2^{d-1}\pi^d}\int_{\mathbb{R}} \int_{|y|\leq\pi}f^2(x, y)dydx.$$  Passing limit $R\to \infty$ leads to the desired estimate (\ref{fxy}).
\end{proof}
Set the error in two scale setting as
$$
e(t, x, y)=\tilde{\Psi}(t, x, y) - \tilde{\Psi}^\varepsilon(t, x, y),
$$
then the error in original variable gives
$$
\Psi(t, x) - \Psi^\epsilon(t, x) = e(t, x, x/\varepsilon).
$$
 Applying Lemma \ref{lem2.1} and using Corrolary \ref{cor2.1} we prove Theorem $1.1$ for the original problem.

\section{Construction}
In this section we  first review the classical asymptotic approach and the band structure,   then the Gaussian beam construction following \cite{DGR06}. For simplicity, the construction and proofs are presented in one-dimensional setting.
\subsection*{Asymptotic Approach}
We look for an approximate solution to (\ref{2.1}) of the form:
\begin{equation}\label{3.1}
\tilde{\Psi}^{\varepsilon}(t,x,y)=A(t,x,y)e^{i\Phi(t,x)/\varepsilon},
\end{equation}
where $$A(t,x,y)=A_0(t,x,y)+A_1(t,x,y)\varepsilon+\dots +A_l(t,x,y)\varepsilon^l,$$
with $A_i$ satisfying:
$$A_i(t,x,y)=A_i(t,x,y+2\pi),\quad i=0,\dots l.$$
\\Then the two-scale Schr\"odinger operator $P$ defined in (\ref{2.11}) when applied upon $\tilde{\Psi}^{\varepsilon}$ gives
$$P(\tilde{\Psi}^\varepsilon)=(c_0+c_1\varepsilon+c_2\varepsilon^2+\dots+c_{l+2}\varepsilon^{l+2})e^{i\Phi/\varepsilon},$$
where by a direct calculation,
\begin{eqnarray}\label{3.2-3.4}
c_0&=&[-\partial_t \Phi -\frac{1}{2}(-i\partial_y+ \partial_x \Phi)^2-V(y)-V_e(x)]A_0=: G(t,x,y)A_0,\label{3.3}
\\c_1&=&i\partial_tA_0+\frac{1}{2}(2\partial_x\cdot\partial_y+2i \partial_x \Phi \cdot\partial_x+i\partial_x^2 \Phi)A_0+G(t,x,y)A_1=:iLA_0+GA_1,
\\c_j&=&\partial_x^2A_{j-2}+iLA_{j-1}+GA_j,\quad j=2,3,\dots,l+2.\label{3.4}
\end{eqnarray}
Here
$$
L:=\partial_t+(-i\partial_y+\partial_x\Phi)\partial_x + \frac{1}{2} \partial_x^2 \Phi.
$$
Observe that, when $\Phi$ is real valued, (\ref{3.1}) is  a standard ansatz of the geometric optics \cite{DGR06}. In the construction of geometric optic solutions it is required that $c_j=0, j=0,1,\dots l+2,$ which gives PDEs for $\Phi, A_0, \cdots, A_l$.  However, $\Phi$ may develop finite time singularities at `caustics' and equations for $A_j$ then become undefined \cite{DGR06}.

\subsection*{Band Structure/Bloch Decomposition}%
The relation $c_0=0$ can be rewritten as
\begin{equation}\label{3.5}
(\Phi_t+H(\partial_x \Phi,y)+V_e(x))A_0=0,
\end{equation}
where $H(k,y)$ with $k=\partial_x \Phi$ is a self-adjoint differential operator, when $k$ is real.
\begin{equation}
H(k,y)=\frac{1}{2}(-i\partial_y+k)^2+V(y).
\end{equation}
We let $z_n$ be the normalized eigenfunction corresponding to $E_n(k)$:
$$
H(k, y)z_n=E_n(k)z_n, \quad \langle z_n, z_n \rangle=1.
$$
From now on we will suppress the index $n,$ since the construction for each band remains the same.

We set the leading amplitude as
 \begin{equation}\label{3.7}
A_0(t,x,y)=a(t,x)z(k(t,x),y),
 \end{equation}
where $k=\partial_x \Phi$, hence (\ref{3.5}) is satisfied as long as $\Phi$ solves the Hamilton-Jacobi equation:
\begin{equation}\label{3.8}
F(t, x):=\partial_t \Phi+E(\partial_x \Phi)+V_e(x)=0.
\end{equation}

\subsection*{A Bloch Decomposition-Based Gaussian Beam Method} Let
$\displaystyle(x, p)=(\tilde{x}(t),p(t))$ be a bicharacteristics of (\ref{3.8}), then
\begin{equation}\label{3.10}
\dot{\tilde{x}}=E^\prime(p),\quad \dot{p}=-V^\prime_e(\tilde{x}).
\end{equation}
From now on, we fix a bi-characteristics $\{(\tilde x(t), p(t)), t>0\}$ with initial data $(x_0, \partial_xS_0(x_0))$ for any $x_0\in K_0=supp_x(g(x,y))$.
We denote by $\gamma$  its projection into the $(x, t)$ space.

The idea underlying the Gaussian beam method is to build asymptotic solutions concentrated on a single ray $\gamma$ so that $\Phi(t, \tilde x(t))$ is real and $Im\{\Phi(t,y)\}>0$ for $y\neq \tilde x(t).$ We are going to choose $\Phi$ so that $Im(\Phi)\geq cd(x,\gamma)^2,$ where $d(x,\gamma)$ is a distance from $x$ to
the central ray  $\gamma$ \cite{JR05}. 
Therefore, instead of solving (\ref{3.8}) exactly, we only need to have $F(x, t)$ vanish to higher order on $\gamma$. For the first order Gaussian beam approximation we choose the phase
$\Phi(t,x)$ a quadratic function:
\begin{equation}\label{3.11}
\Phi(t,x)=S(t)+p(t)(x-\tilde{x}(t))+\frac{1}{2}M(t)(x-\tilde{x}(t))^2.
\end{equation}
With this choice we have
\begin{equation}\label{3.12}
F(t, x)=\dot{S}+\dot{p}(x-\tilde{x})-p\dot{\tilde{x}}+\frac{1}{2}\dot{M}(x-\tilde{x})^2-M(x-\tilde{x})\dot{\tilde{x}}+E(p+M(x-\tilde{x}))+V_e(x).
\end{equation}
We see that $F(t, \tilde x(t))=0$ gives the evolution equation for $S$,
$$
\dot{S}=pE^\prime(p)-E(p)-V_e(\tilde x).
$$
It can be verified $\partial_x F(t, \tilde x(t))=0$ is equivalent to $\dot{p}=-V^\prime_e(\tilde{x})$, which is the second equation in (\ref{3.10}).  From $\partial_x^2F(t, \tilde x(t))=0$ we obtain the equation
for $M$:
 \begin{equation}\label{3.13}
 \dot{M}=-E^{\prime\prime}(p)M^2-V^{\prime\prime}_e(\tilde{x}).
 \end{equation}
It is clear that we should set initial condition for the phase as
\begin{equation}
S(0)=S_0(x_0),
\end{equation}
where $S_0$ is a given initial phase in (\ref{1.2}). Note that equation (\ref{3.13}) is a nonlinear Ricatti type equation. The important result about $M$ is given in \cite{DGR06}, proving that global solution for $M$ exists and $Im(M)$ remains positive (positive definite in multi-dimensional setting) for all time $t$ as long as $Im(M(0))$ is positive. Therefore we choose
\begin{equation}
M(0)=\partial_x^2S_0(x_0)+i,
\end{equation}
which satisfies $Im(M(0))>0$ as required in the Gaussian beam approximation.

It follows from our construction that $c_0$ vanishes up to third order on $\tilde{x}.$  In fact,
\begin{align}\label{c0}
c_0 & =G(az(k(t, x), y))\\ \notag
 &=a(t, x)F(t, x)z(k(t, x), y)\\ \notag
 &=\frac{a(t, x)}{3!}\partial_x^3 F(t, x^*)z(k(t, x), y)(x-\tilde x)^3,
\end{align}
where $x^*$ is an intermediate value between $x$ and $\tilde x$.  A simple calculation gives
\begin{equation}\label{3.17}
\partial_x^3F(t, x^*)=(V_e^{(3)}(x^*)+E^{(3)}(p+ M(x^*-\tilde x ))M^3(t)),
\end{equation}
which is uniformly bounded near the ray $\tilde x$ since $V_e\in C_b^5(\mathbb{R})$ and (\ref{1.5}) holds.  Hence $c_0$ will be bounded by $O(|x-\tilde x|^3)$ as long as the amplitude is bounded.

\subsection*{Equation for the Amplitude}  For the first order Gaussian beam construction, we shall determine the amplitudes  so that $c_1$ vanishes to the first order on $\gamma$. Note that
$$c_1=iLA_0+GA_1,$$
where
$$
G=-(\Phi_t+H(k,y)+V_e(x))= - F(t, x) +E(k)-H(k, y).
$$
On the ray $x=\tilde x(t)$, we require that $c_1=0$, that is
$$
iLA_0 +(E(p)-H(p, y))A_1=0.
$$
In order for $A_1$ to exist,  it is necessary that
\begin{equation}\label{A_0}
\langle LA_0,  z  \rangle|_{x=\tilde x (t)}=0.
\end{equation}
For $x\not=\tilde x(t)$, we have
$$
c_1=iLA_0-FA_1 +(E(k)-H(k, y))A_1^\top,
$$
where $A_1^\top$ contains the orthogonal compliment of $z$, satisfying $\langle A_1^\top, z\rangle=0$. 
We let
\begin{equation}\label{A_1}
A_1^\top =i (E(k)-H)^{-1} [\langle LA_0,z\rangle{z}-L(A_0)].
\end{equation}
Therefore  using (\ref{A_0}) and Taylor expansion at $\tilde x$,
\begin{equation}\label{c1}
c_1=i\langle LA_0,z\rangle -FA_1=i\partial_x \langle LA_0, z\rangle(t, x^*)(x-\tilde x)-FA_1.
\end{equation}
With further refined calculation,  (\ref{A_0}) and  (\ref{c1}) yield the following result.
\begin{lemm}\label{lem3.1}
 For the first order Gaussian beam construction,  $a(t, x)=a(t;x_0)$  and satisfies the following evolution equation along the ray $x=\tilde x(t)$:
\begin{equation}
a_t=a\Big(V_e^\prime(\tilde{x})\langle \partial_k z(p, \cdot), z(p, \cdot)\rangle-\frac{1}{2}E^{\prime\prime}(p)M\Big).
\end{equation}
Moreover, for $x\not=\tilde x(t)$ we have
\begin{align}\label{3.22+}
c_0 & =\frac{a(t;x_0)}{3!}\partial_x^3F(t,x^*) z(k, y)(x-\tilde x)^3,\\ \label{3.22}
c_1 & =-ia \langle \partial_k z(p, \cdot), z(p, \cdot)\rangle(E^{\prime\prime}(p)M^2+V_e^{\prime\prime}(\tilde{x}))(x-\tilde{x})-F(t, x)A_1, \\ \label{3.22++}
c_2 &=a(t;x_0)M^2\partial_k^2z(k, y)+iLA_1,
\end{align}
where $A_1\in \text{span}\{A_1^\top,  z\}$.
\end{lemm}
\begin{proof}
Recall that
$$
A_0=az(k(t, x), y), \quad k(t,x)=p(t)+M(t)(x-\tilde{x}(t))
$$
and
$$
L=\partial_t+H_k(k,y) \partial_x+\frac{1}{2}\partial_x^2\Phi=\partial_t+H_k(k,y) \partial_x+\frac{1}{2}M.
$$
We take $a(t, x)=a(t; x_0)$, and calculate
\begin{align*}
\langle L(az),z\rangle &=\partial_t a+ \frac{1}{2} aM + a \langle \partial_t z, z\rangle +a  \langle H_k \partial_x z,z\rangle \\
&=\partial_t a+ a\left( \frac{1}{2} M + k_t \langle \partial_k z, z\rangle +  k_x  \langle H_k \partial_k z,z\rangle \right).
\end{align*}
We  observe that the eigenvalue identity $\displaystyle Hz =Ez$ holds for any $k$, implying
$$
H_{kk}z+2H_k \partial_k z +H \partial_k^2 z =E^{\prime\prime}(k)z+2E^\prime \partial_k z +E \partial_k^2 z.
$$
This  against $z$  using $H_{kk}=1$ and $\langle(H-E)\partial_k^2 z,z\rangle=0$ leads to
$$
E^{\prime\prime}(k) =1+2\langle H_k\partial_k z,z\rangle-2E^\prime\langle \partial_k z,z\rangle.
$$
Hence using $k_x=M$ we have
$$
 \frac{1}{2} M  +k_x \langle H_k \partial_k z,z\rangle=\frac{1}{2}E^{\prime\prime}(k) M +E^\prime M \langle \partial_k z,z\rangle.
$$
Putting together we obtain
\begin{align*}
\langle L(az),z\rangle =\partial_t a+ a\left( \frac{1}{2}E^{\prime\prime}(k) M +  (k_t + E^\prime M)  \langle \partial_k z, z\rangle  \right),
\end{align*}
where
$$
k_t=  - V_e'(\tilde x) - E'(p)M+\dot M(x-\tilde x(t)).
$$
Thus (\ref{A_0}) gives the desired amplitude equation.  Recalling (\ref{c0}) and (\ref{3.17}) we  have (\ref{3.22+}).
(\ref{c1}) yields
$$
c_1=ia\dot M  \langle \partial_k z, z\rangle (x-\tilde x)-FA_1,
$$
which in virtue of (\ref{3.13}) gives (\ref{3.22}).  From (\ref{3.4}) it follows that
$$
c_2=\partial^2_x (az) +iLA_1=a(k_x)^2\partial_{k}^2z+iLA_1
$$
which gives (\ref{3.22++}).
\end{proof}
Therefore, the system of ODEs for GB components is set up:
\begin{equation}\label{3.23}
\begin{cases} \dot{\tilde{x}}=E^\prime(p),\quad \tilde{x}|_{t=0}=x_0, \\
\dot{p}=-V^\prime_e(\tilde{x}),\quad p|_{t=0}=\partial_xS_0(x_0), \\
\dot{S}=pE^\prime(p)-E(p)-V_e(\tilde{x}),\quad S|_{t=0}=S_0(x_0), \\
\dot{M}=-E^{\prime\prime}(p)M^2-V^{\prime\prime}_e(\tilde{x}),\quad M|_{t=0}=\partial_{x}^2S_0(x_0)+i,\\
\dot{a}=a(V_e^\prime(\tilde{x}) \langle \partial_k z(p, \cdot), z(p, \cdot)\rangle -\frac{1}{2}E^{\prime\prime}(p)M),\quad a|_{t=0}=a(x_0),\\
\end{cases}
\end{equation}
where  the initial value for the amplitude $a(t;x_0)$ is taken as
\begin{equation}
a|_{t=0}=a(x_0)=\int_0^{2\pi}g(x,y)\overline{z(\partial_x S_0,y)}dy.
\end{equation}

\begin{rmk}
For the derivation of the equations for the Gaussian beam components for the higher order approximations, we refer the reader to \cite{DGR06}.
\end{rmk}
In order to complete the estimate for $c_i$, we still need to estimate $A_1$.  The following result will be used later in the estimate of the evolution error.

 \begin{lemm}\label{lem3.2-}
For any positive integer $m,$  each eigenvector $z_n(k,y)$ satisfies the following condition:
\begin{equation}\label{326}
\sum_{|\beta_1| \leq m, |\beta_2| \leq 3}  \|\partial_k^{\beta_1}\partial_y^{\beta_2} z_n(k, y)\|_{L^2_y}\leq Z<\infty.
\end{equation}
\end{lemm}
\begin{proof}  For every fixed $k$, let $(z(k, y), E(k))$ be an eigen-pair that satisfies  the eigenvalue equation (\ref{1.4}), i.e,
\begin{equation}\label{he}
H(k, y)z(k, y)= \left( \frac{1}{2} (-i\nabla _y+k)^2 +V(y) \right)z(k, y)=E(k)z(k, y), \quad y\in \Gamma.
\end{equation}
Since $v\in C^2(\Gamma)$ and $E \in L^\infty$, by the elliptic regularity theory \cite[Theorem 6.19]{GT98}, $z(k,y)\in C^3(\Gamma)$ in $y$ variable, which gives (\ref{326}) for $m=0$.

We next prove (\ref{326}) by induction. We assume that $\partial_k^\beta z \in C^3(\Gamma)$ for $|\beta|\leq l-1$ with  $1\leq l\leq m-1$.
Note that $\partial_k^\alpha H=0$ for any $\alpha$ with $|\alpha| \geq 3$, then differentiation of (\ref{he}) to higher order,  using the general Leibnitz rule, gives
\begin{align}\label{hea}
(H(k, y)-E(k))\partial_k^\beta z & =\sum_{\alpha <\beta} \left(\begin{array}{c}
\beta \\
\alpha
\end{array}\right) (\partial_k^{\beta-\alpha} E(k))(\partial_k^\alpha z) \\ \notag
& \quad -\sum_{\alpha <\beta, |\alpha|\geq |\beta|-2} \left(\begin{array}{c}
\beta \\
\alpha
\end{array}\right)
(\partial_k^{\beta-\alpha} H(k, y))(\partial_k^\alpha z).
\end{align}
The same elliptic regularity theory when applied to (\ref{hea}) yields
$$
\partial_k^\beta z \in C^3(\Gamma), \quad |\beta|=l.
$$
Here the needed $\partial_k^{\beta-\alpha}E \in L^\infty$ is ensured again by assumption (\ref{1.5}).  The proof of  (\ref{326}) is complete.
\end{proof}

\begin{lemm}\label{lem3.2} With the eigenvector $z(k,y)$ satisfying (\ref{326}),
we have that for $\alpha=0,1,$
\begin{equation}\label{a1}
\sup_{t,x_0}\int_0^{2\pi}|L^\alpha A_1|^2dy\leq CZ(1+Z+Z^2),
\end{equation}
where $C$ depends on the spectral gap $\displaystyle\Delta E=\min_{i\neq j}|E_i-E_j|>0$ and the Gaussian beam components.
\end{lemm}
\begin{proof}
Since $A_1$ is a linear combination of $A_1^\top$ and $z$, we will prove (\ref{a1}) for $A_1^\top$ only. Set
$$
B: =i(LA_0-\langle LA_0,z\rangle z),
$$
we have
\begin{equation}\label{a1b}
A_1^\top=(H-E)^{-1}B.
\end{equation}
We proceed in two steps:  \\

{\sl Step 1.}   Estimate of $LA_1^\top$ in terms of $B$. \\

A careful calculation gives that
\begin{equation}\label{LA}
(H-E)LA_1^\top=LB-k_t(H_k-E_k)A_1^\top-k_xH_k(H_k-E_k)A_1^\top + iV'(y)\partial_xA_1^\top.
\end{equation}
In fact,   applying $L$ to  (\ref{a1b})  gives
$$
(H_k-E_k)k_tA_1^\top +(H-E)\partial_t A_1^\top+H_k\partial_x[(H-E)A_1^\top]+\frac{1}{2}k_x(H-E)A_1^\top=LB.$$
Note that  $$\partial_x[(H-E)A_1]=(H_k-E_k)k_xA_1+(H-E)\partial_xA_1.$$
Using the definition of operators $H$ and $H_k$ we also have
$$
H_k(H-E)=(H-E)H_k-iV'(y).
$$
These together verifies (\ref{LA}).

From (\ref{LA}) it follows that
\begin{equation}\label{la1b}
\|LA_1^\top\|_{L^2_y}\leq \frac{C}{\Delta E} \left(\|LB\|_{L^2_y} +\sum_{j=0}^2\|\partial_y^jA_1^\top\|_{L^2_y}+\|\partial_x A_1^\top\|_{L^2_y}\right),
\end{equation}
here $C$ depends on $k_t, k_x, E_k$ and $V'(y)$, and we have used the following resolvent estimate,
$$
\displaystyle\|(H-E)^{-1}\|_{L^2}\leq\frac{1}{\Delta E},
$$
where the domain of the operator $(H-E)^{-1}$ is restricted to the orthogonal complement of the eigenvector $z.$
Next we estimate the right hand of (\ref{la1b}) in terms of $B$. From here on we use $C$ to denote a generic constant  depending on $\Delta E$,  $k, E, z$ and their derivatives.
 We note that
$$LB=B_t+H_kB_x+\frac{1}{2}k_xB=B_t-iB_{xy}+kB_x+\frac{1}{2}k_xB,
$$
which yields
$$
\|LB\|_{L^2_y}\leq C(\|B\|_{L^2_y}+\|B_t\|_{L^2_y}+\|B_x\|_{L^2_y}+\|B_{xy}\|_{L^2_y}).
$$
In the rest of this proof, we shall use $\|\cdot\|$ to denote $\|\cdot\|_{L^2_y}$.

From (\ref{a1b}) it follows that
\begin{align*}
B_y & =(H-E)A_{1y}-iV'(y)A_1, \\
B_{yy}& =(H-E)A_{1yy}-2iV'(y)A_{1y}-iV^{\prime\prime}(y)A_1,\\
B_x & =(H-E)A_{1x}+k_x(H_k-E_k)A_1.
\end{align*}
Again from (\ref{a1b}) we obtain $\|A_1^\top\| \leq C\|B\|$, which when combined with the above gives
\begin{align*}
\|A_{1y}\| &  \leq C(\|B\|+ \|B_y\|), \\
\|A_{1yy}\| & \leq C\left(\|B_{yy}\|+\|A_{1y}\|+\|A_1\|\right)\leq C \sum_{j=0}^2\|\partial_y^j B\|, \\
\|A_{1x}\| & \leq C \left(\|B_x\|+\|A_{1y}\|+\|A_1\|\right)\leq C\left(\|B\|+\|B_y\|+\|B_x\|\right).
\end{align*}
Therefore,
\begin{equation}\label{B}
\|LA_1^\top \|\leq C(\|B\|+ \|B_t\|+\|B_x\|+\|B_{xy}\|+\|B_y\|+\|B_{yy}\|).
\end{equation}

{\sl Step 2.}   Estimate of $B$. \\

Note that
\begin{align*}
B&=L(az)-\langle L(az),z\rangle z\\
& =ak_t(z_k-\langle z_k,z\rangle z)+ak_x(H_kz_k-\langle H_kz_k,z\rangle z) \\
& =ak_t \tilde f_1 +ak_x \tilde f_2,
\end{align*}
where $\tilde f_i$ are of the form
$$
\tilde{f}(k,y)=f(k,y)-\langle f(k,\cdot), z(k,\cdot)\rangle z(k,y),
$$
with $f_1=z_k$ and $f_2=H_kz_k=-iz_{ky} +k z_k$.   The right hand side of (\ref{B})  is majored by
$$
I_1+I_2: = C \sum_{i=1}^2 ( (\|\tilde f_i\|+ \|\partial_t \tilde f_i\|+\|\partial_x \tilde f_i\| +\|\partial_{xy}^2 \tilde f_i \|+\|\partial_y \tilde f_i \|+\|\partial_y^2 \tilde f_i \|).
$$
We apply Lemma \ref{f} below to bound both $I_1$ and $I_2$.
 \begin{align*}
I_1& \leq C(\|z_k\| +\|z_{kk}\|+\|z_k\|^2+(1+\|z_k\|)(\|z_{ky}\| + \|z_{kyy})\|) +\|z_{kky}\|\\
& \qquad \|z_{kk}\|\|z_y\| +\|z_{ky}\|\|z_k\| +\|z_k\|^2\|z_y\|) \\
& \leq CZ(1+Z+Z^2).
\end{align*}
Since $f_2=-iz_{ky}+kf_1$, it suffices to bound $I_2$ by considering only $f_2=z_{ky}$. By Lemma  \ref{f} we have
\begin{align*}
I_2& \leq C(\|z_{ky} \| + \|z_{kky}\|  +\|z_{ky}\|\|z_k\|\\
&\quad +\|z_{kyy}\|+\|z_{kyyy}\| +\|z_{ky}\| (1+\|z_y\|+\|z_{yy}\|)\\
&\quad +\|z_{kkyy}\| +\|z_{kky}\|\|z_y\| +\|z_{kyy}\|\|z_k\|+ \|z_{ky}\|^2 +\|z_{ky}\|\|z_k\|\|z_y\|) \\
&  \leq CZ(1+Z+Z^2).
\end{align*}
These together with (\ref{B}) yield
$$
\|LA_1^\top\|\leq CZ(1+Z+Z^2).
$$
This proves the boundedness of $\|LA_1\|.$
 \end{proof}

 \begin{lemm}\label{f}
Let $f(k,y)$ be smooth and integrable in $y$ and
\begin{equation}\label{fz}
\tilde{f}(k,y)=f(k,y)-\langle f(k,\cdot),z(k,\cdot)\rangle z(k,y).
\end{equation}
Then for $k=k(t, x)$ the following estimates hold:
\begin{enumerate}
\item $\|\tilde{f}_t, \tilde{f}_x \|\leq C(\|f_k\|+\|f\|\|z_k\|)$,\\
\item $\displaystyle\|\partial_y^j\tilde{f}\|\leq \|\partial_y^j f\|+\|f\|\|\partial_y^j z\|,\quad j=1,2,$\\
\item $\|\tilde{f}_{xy}\|\leq C(\|f_{ky}\|+\|f_k\|\|z_y\|+\|f_y\|\|z_k\|+\|f\|\|z_{ky}\|+\|f\|\|z_k\|\|z_y\|),$
\end{enumerate}
where constant $C$ depends on $k_t$ and $k_x$, and the norm $\|\cdot\|:=\|\cdot\|_{L^2_y}$.
\end{lemm}
\begin{proof}
By the chain rule,
$$\tilde{f}_t=k_tf_k-k_t\langle f_k,z\rangle z-\overline{k_t}\langle f,z_k\rangle z-k_t\langle f,z\rangle z_k.$$
Using the Cauchy inequality together with the fact that $z$ is normalized, we obtain
$$\|\tilde{f}_t\|\leq C(\|f_k\|+2\|f\|\|z_k\|).$$
Same estimate follows for $f_x$.

For differentiation in $y$ we have
$$ \partial_y^j \tilde{f}= \partial_y^j f -\langle f, z\rangle \partial^j_y z,$$
leading to
$$\|\partial_y^j \tilde{f} \|\leq\|\partial_y^j f\|+\|f\|\|\partial_y^j z\|.
$$
Finally,
$$
\tilde{f}_{xy}=k_xf_{ky}-k_x\langle f_k,z\rangle z_y-\overline{k_x}\langle f,z_k\rangle z_y-k_x \langle f,z\rangle z_{ky}.
$$
Hence
$$
\|\tilde{f}_{xy}\|\leq C (\|f_{ky}\|+\|f_k\|\|z_y\|+\|f\|\|z_k\|\|z_y\|+\|f\|\|z_{ky}\|)
$$
which concludes the proof of the lemma.
\end{proof}

\subsection*{Gaussian Beam Superposition and Residuals}
We solve ODE system (\ref{3.23}) for each band, and obtain a band based Gaussian beam approximation along a given ray:
\begin{equation}\label{3.25}
\tilde{\Psi}_{GB}^{\varepsilon n}(t,x,y;x_0)=(a_n(t;x_0)z_n(k_n,y)+\varepsilon A_1^n(t,x,y;x_0))e^{i\Phi_n(t,x;x_0)/\varepsilon}.
\end{equation}
Since the Schr\"odinger equation is linear, the approximate solution can be generated by a superposition of neighboring Gaussian beams and over all available bands
\begin{equation}
\tilde{\Psi}^\varepsilon(t,x,y) =\frac{1}{\sqrt{2\pi\varepsilon}}\int_{K_0}\sum_{n=1}^N \tilde{\Psi}_{GB}^{\varepsilon n}(t,x,y,x_0)dx_0,
\end{equation}
where $\displaystyle\frac{1}{\sqrt{2\pi\varepsilon}}$ is a normalized constant chosen to match initial data against the Gaussian profile.
Let us use the notation
\begin{equation}\label{3.24}
\tilde{\Psi}^{\varepsilon n}(t,x,y):= \frac{1}{\sqrt{2\pi\varepsilon}}\int_{K_0}\tilde{\Psi}_{GB}^{\varepsilon n}(t,x,y;x_0)dx_0,
\end{equation}
then Lemma \ref{lem3.1} yields the following residual representation:
\begin{equation}\label{3.28}
P(\tilde{\Psi}^{\varepsilon n})=\frac{1}{\sqrt{2\pi\varepsilon}}\int_{K_0} \left( c_{0n}+ \varepsilon c_{1n}+ \varepsilon^2 c_{2n} \right) e^{i\Phi_n(t,x;x_0)/\varepsilon} dx_0.
\end{equation}
In next two sections we provide proofs of the  accuracy results. We start with the initial error estimation.

\section{Initial Error - proof of Theorem \ref{th2.1}.}
{In this section,  the unmarked norm $\|\cdot\|$ denotes $\|\cdot\|_{L^2_{x,y}}$- norm unless otherwise specified.}

For simplicity of presentation, we only give the one dimensional estimate with $d=1$.
The initial phase can be expressed as
$$S_0(x)=S_0(x_0)+S_0^\prime(x_0)(x-x_0)+S_0^{\prime\prime}(x_0)\frac{(x-x_0)^2}{2}+R_2^{x_0}[S_0]=T_2^{x_0}[S_0](x)+R_2^{x_0}[S_0](x),
$$
where
$$\displaystyle R_2^{x_0}[S_0]=\frac{|S_0^{(3)}(\eta(x,x_0))|(x-x_0)^3}{3!}
$$ is the remainder of the Taylor expansion.
The idea of the proof of Theorem \ref{th2.1}  is to introduce
\begin{equation}\label{4.1}
\Psi^*=\frac{1}{\sqrt{2\pi\varepsilon}}\int_\mathbb{R}g(x_0,y)e^{iT_2^{x_0}[S_0](x)/\varepsilon}e^{-\frac{(x-x_0)^2}{2\varepsilon}}dx_0,
\end{equation}
so that
\begin{equation}\label{pp}
\|\tilde{\Psi}_0-\tilde{\Psi}_0^\varepsilon\|\leq\|\tilde{\Psi}_0-\Psi^*\|+\|\Psi^*-\tilde{\Psi}_0^\varepsilon\|,
\end{equation}
where the initial condition $ \tilde{\Psi}_0=\tilde{\Psi}(0, x, y)$ defined in (\ref{2.1}),
$\tilde{\Psi}_0^\varepsilon=\tilde{\Psi}^\varepsilon(0, x, y)$ defined in (\ref{2.8}) is the the Gaussian beam superposition evaluated at $t=0,$
\begin{equation} \label{4.3}
\tilde{\Psi}_0^{\varepsilon}=\frac{1}{\sqrt{2\pi\varepsilon}}\int_{K_0}\sum_{n=1}^N(a_n(x_0)z_n(\partial_x\Phi^0(x;x_0),y)+\varepsilon A_1^n(0,x,y;x_0))e^{i\Phi^0(x;x_0)/\varepsilon}dx_0,
\end{equation}
where  from (\ref{2.9}) we have
$$\Phi^0(x,x_0)=T_2^{x_0}[S_0](x)+\frac{i(x-x_0)^2}{2}.
$$

The rest of this section is to estimate two terms on the right of  (\ref{pp}), which will  be given  in Lemma 4.1 and Lemma 4.2 below, respectively.
\begin{lemm}\label{lem4.1}
Let $\Psi^*$ be defined in (\ref{4.1}), $g(x,y)\in H^1(K_0\times[0,2\pi]),$ then
$$\|\Psi^*-\tilde{\Psi}_0\|\leq \Big(\|\partial_xg\|+\sqrt{\frac{5}{12}}\max_{x\in \mathbb{R}} |S_0^{(3)}(x)|\|g\|\Big)\varepsilon^{1/2}.$$
\end{lemm}
\begin{proof}
Using that $$\frac{1}{\sqrt{2\pi\varepsilon}}\int_\mathbb{R}e^{-\frac{(x-x_0)^2}{2\varepsilon}}dx_0=1,$$
\begin{equation}\Psi^*-\tilde{\Psi}_0=\frac{1}{\sqrt{2\pi\varepsilon}}\int_\mathbb{R}[g(x_0,y)e^{iT_2^{x_0}[S_0](x)/\varepsilon}-g(x,y)e^{iS_0(x)/\varepsilon}]e^{-\frac{(x-x_0)^2}{2\varepsilon}}dx_0=I+J,
\end{equation}
where
$$I=\frac{1}{\sqrt{2\pi\varepsilon}}\int_\mathbb{R}(g(x_0,y)-g(x,y))e^{iT_2^{x_0}[S_0](x)/\varepsilon}e^{-\frac{(x-x_0)^2}{2\varepsilon}}dx_0,$$
$$J=\frac{1}{\sqrt{2\pi\varepsilon}}\int_\mathbb{R}g(x,y)(e^{iT_2^{x_0}[S_0](x)/\varepsilon}-e^{iS_0(x)/\varepsilon})e^{-\frac{(x-x_0)^2}{2\varepsilon}}dx_0.$$
Our next step is to find estimates for $\|I\|$ and $\|J\|.$
\begin{eqnarray*}
\|I\|^2&=&\frac{1}{2\pi\varepsilon}\Big\|\int_\mathbb{R}(g(x_0,y)-g(x,y))|e^{iT_2^{x_0}[S_0](x)/\varepsilon}|e^{-\frac{(x-x_0)^2}{2\varepsilon}}dx_0\Big\|\\
&=&\frac{1}{2\pi\varepsilon}\int_0^{2\pi}\int_\mathbb{R}\Big[\int_\mathbb{R}(g(x_0,y)-g(x,y))e^{-\frac{(x-x_0)^2}{2\varepsilon}}dx_0\Big]^2dxdy.
\end{eqnarray*}
For fixed $x,$ we introduce a new variable $\displaystyle\xi=\frac{x-x_0}{\sqrt{2\varepsilon}},\quad dx_0=-\sqrt{2\varepsilon}d\xi$\quad to obtain
$$\|I\|^2=\frac{1}{\pi}\int_0^{2\pi}\int_\mathbb{R}\Big(\int_\mathbb{R}|g(x-\sqrt{2\varepsilon}\xi,y)-g(x,y)|e^{-\xi^2}d\xi\Big)^2dxdy.$$
By the H\"older inequality,
\begin{eqnarray*}
\|I\|^2&\leq&\frac{1}{\pi}\int_0^{2\pi}\int_\mathbb{R}\int_\mathbb{R}|g(x-\sqrt{2\varepsilon}\xi,y)-g(x,y)|^2e^{-\xi^2}d\xi\int_{R}e^{-\xi^2}d\xi dxdy\\
&=&\frac{1}{\sqrt{\pi}}\int_0^{2\pi}\int_{\mathbb{R}}\int_{\mathbb{R}}|g(x-\sqrt{2\varepsilon}\xi,y)-g(x,y)|^2e^{-\xi^2}d\xi dxdy.
\end{eqnarray*}
Using the mean value theorem for $g,$ we have $$g(x-\sqrt{2\varepsilon}\xi, y)-g(x, y)=- \partial_x g(x-\eta^*\sqrt{2\varepsilon}\xi, y)\sqrt{2\varepsilon}\xi=- \partial_x
g(x-\eta^*(x-x_0), y)\sqrt{2\varepsilon}\xi.$$
\\Hence,
\begin{eqnarray*} \|I\|^2&\leq&\frac{2\varepsilon}{\sqrt{\pi}}\int_0^{2\pi}\int_\mathbb{R}\int_\mathbb{R}|\partial_xg(x-\eta^*(x-x_0),y)|^2dx\xi^2e^{-\xi^2}d\xi dy\\
&=&\frac{2\varepsilon}{\sqrt{\pi}}\|\partial_xg\|^2\int_{\mathbb{R}}\xi^2e^{-\xi^2}d\xi =\varepsilon\|\partial_xg\|^2.
\end{eqnarray*}
Now we turn to the estimation of $\displaystyle\|J\|:$
\begin{eqnarray*}
\|J\|^2&=&\frac{1}{2\pi\varepsilon}\Big\|\int_\mathbb{R}g(x,y)(e^{iT_2^{x_0}[S_0](x)/\varepsilon}-e^{iS_0(x)/\varepsilon})e^{-\frac{(x-x_0)^2}{2\varepsilon}}dx_0\Big\|^2\\
&\leq&\frac{1}{2\pi\varepsilon}\int_0^{2\pi}\int_\mathbb{R}\Big[\int_\mathbb{R}|g(x,y)||e^{iS_0(x)/\varepsilon}||e^{-iR_2^{x_0}[S_0](x)/\varepsilon}-1|e^{-\frac{(x-x_0)^2}{2\varepsilon}}dx_0\Big]^2dxdy.
\end{eqnarray*}
Since $S_0$ is real, $|e^{iS_0(x)/\varepsilon}|=1.$ The above is further  bounded by
$$
\frac{1}{2\pi\varepsilon}\int_0^{2\pi}\int_\mathbb{R}\Big[\int_\mathbb{R}|g(x,y)|\Big[\Big(\cos\frac{R_2^{x_0}[S_0](x)}{\varepsilon}-1\Big)^2+\sin^2\frac{R_2^{x_0}[S_0](x)}{\varepsilon}\Big]^{1/2}e^{-\frac{(x-x_0)^2}{2\varepsilon}}dx_0\Big]^2dxdy.$$
Using a half-angle formula for $\sin x$ and that $|\sin x|\leq|x|,$ we obtain:
\begin{eqnarray*}
\|J\|^2&\leq&\frac{1}{2\pi\varepsilon}\int_0^{2\pi}\int_\mathbb{R}\Big[\int_\mathbb{R}|g(x,y)|\Big(4\sin^2\frac{R_2^{x_0}[S_0](x)}{2\varepsilon}\Big)^{1/2}e^{-\frac{(x-x_0)^2}{2\varepsilon}}dx_0\Big]^2dxdy\\
&\leq&\frac{1}{2\pi\varepsilon}\int_0^{2\pi}\int_\mathbb{R}\Big[\int_\mathbb{R}|g(x,y)|\frac{R_2^{x_0}[S_0](x)}{\varepsilon}e^{-\frac{(x-x_0)^2}{2\varepsilon}}dx_0\Big]^2dxdy.
\end{eqnarray*}
Using the remainder formula and the H\"older inequality,
$$\|J\|^2\leq\frac{1}{2\pi\varepsilon}\int_0^{2\pi}\int_\mathbb{R}\int_\mathbb{R}|g(x,y)|^2e^{-\frac{(x-x_0)^2}{2\varepsilon}}dx_0\int_\mathbb{R}\frac{(|S_0^{(3)}(\eta)|)^2}{36}\frac{|x-x_0|^6}{\varepsilon^2}e^{-\frac{(x-x_0)^2}{2\varepsilon}}dx_0dxdy.$$
Now, applying the same change of variable as for the term $I,$   $\displaystyle\xi=\frac{x-x_0}{\sqrt{2\varepsilon}},$ ($x$ variable is fixed) we get:
\begin{eqnarray*}
\|J\|^2&\leq&\frac{1}{2\pi\varepsilon}\|g\|^2\frac{(\max |S_0^{(3)}(x)|)^2}{36}\sqrt{2\pi\varepsilon}\int_\mathbb{R}8\sqrt{2} \varepsilon^{-2+3+1/2}|\xi|^6e^{-\xi^2}d\xi\\ &\leq&\frac{\sqrt{2\pi}\max_{x\in \mathbb{R}} |S_0^{(3)}(x)|^2 }{72\pi} \times 8\sqrt{2}\times \frac{15}{8}\sqrt{\pi}\|g\|^2 \varepsilon \\
&=&\frac{5}{12}\max_{x\in \mathbb{R}} |S_0^{(3)}(x)|^2\|g\|^2\varepsilon.
\end{eqnarray*}
Hence, summing both parts, we conclude that:
\begin{eqnarray*}
\|\Psi^*-\tilde{\Psi}_0\|&\leq&\|I\|+\|J\|\\
&\leq&\Big(\|\partial_xg\|+\sqrt{\frac{5}{12}}\max_{x\in\mathbb{R}} |S_0^{(3)}(x)|\|g\|\Big)\varepsilon^{1/2}.
\end{eqnarray*}
\end{proof}
Our next step is to find an estimate for the difference between GB ansatz and $\Psi^*.$
\begin{lemm}\label{4.2}
The following estimate holds:
$$\|\tilde{\Psi}_0^{\varepsilon}-\Psi^*\|\leq C\varepsilon^{1/2},$$
where
$$C=2\pi\max_{k,1\leq n\leq N} \| \partial_kz_n(k, y)\|_{L^\infty_{y}}^2\int_{K_0}\Big|\sum_{n=1}^Na_n(x_0)\Big|^2(S_0^{\prime\prime2}(x_0)+1)dx_0,$$
can be computed from the initial data.
\end{lemm}
\begin{proof} According to our construction,
\begin{align*}
\|\tilde{\Psi}_0^{\varepsilon}-\Psi^*\|^2&=\Big\|\tilde{\Psi}_0^{\varepsilon}-\frac{1}{\sqrt{2\pi\varepsilon}}\int_\mathbb{R}g(x_0,y)e^{iT_2^{x_0}[S_0](x)/\varepsilon}e^{-\frac{(x-x_0)^2}{2\varepsilon}}dx_0\Big\|^2\\
&=\frac{1}{2\pi\varepsilon}\int_0^{2\pi}\int_\mathbb{R}\Big|\int_{K_0}\sum_{n=1}^N(a_n(x_0)(z_n(\partial_x\Phi^0(x,x_0),y)-z_n(\partial_xS_0(x_0),y))\\
&+\varepsilon A_1^n(0,x,y;x_0))e^{\frac{i\Phi^0(x,x_0)}{\varepsilon}}dx_0\Big|^2dxdy.\\
\end{align*}
Then, putting the absolute value sign inside the integral over $K_0$, we observe that
\begin{align*}
\|\tilde{\Psi}_0^{\varepsilon}-\Psi^*\|^2&\leq\frac{1}{2\pi\varepsilon}\int_0^{2\pi}\int_\mathbb{R}\Big[\int_{K_0}\Big|\sum_{n=1}^N(a_n(x_0)(z_n(\partial_x\Phi^0(x,x_0),y)-z_n(\partial_xS_0(x_0),y))\\
&+\varepsilon A_1^n(0,x,y;x_0))\Big| e^{-\frac{(x-x_0)^2}{2\varepsilon}}dx_0\Big]^2dxdy.
\end{align*}
By the H\"older inequality,
\begin{align*}
\|\tilde{\Psi}_0^{\varepsilon}-\Psi^*\|^2&\leq\frac{1}{2\pi\varepsilon}\int_0^{2\pi}\int_\mathbb{R}\int_{K_0}\Big|\sum_{n=1}^N(a_n(x_0)(z_n(\partial_x\Phi^0(x,x_0),y)-z_n(\partial_xS_0(x_0),y))\\
&+\varepsilon A_1^n(0,x,y;x_0))\Big|^2e^{-\frac{(x-x_0)^2}{2\varepsilon}}dx_0\int_{K_0}e^{-\frac{(x-x_0)^2}{2\varepsilon}}dx_0dxdy\\
&=\frac{1}{\sqrt{2\pi\varepsilon}}\int_0^{2\pi}\int_\mathbb{R}\int_{K_0}\Big|\sum_{n=1}^N(a_n(x_0)(z_n(\partial_x\Phi^0(x,x_0),y)-z_n(\partial_xS_0(x_0),y))\\
&+\varepsilon A_1^n(0,x,y;x_0))\Big|^2e^{-\frac{(x-x_0)^2}{2\varepsilon}}dx_0dxdy.
\end{align*}
Using that
\begin{align*}
z_n(\partial_x\Phi^0(x,x_0),y)-z_n(\partial_xS_0(x_0),y)=&z_n(S_0^\prime(x_0)+(S_0^{\prime\prime}(x_0)+i)(x-x_0),y)-z_n(S_0^\prime(x_0),y)\\
=&\partial_kz_n(\eta_n(x,x_0),y)(S_0^{\prime\prime}(x_0)+i)(x-x_0),
\end{align*}
we obtain:
\begin{align*}
\|\tilde{\Psi}_0^{\varepsilon}-\Psi^*\|^2&\leq\frac{1}{\sqrt{2\pi\varepsilon}}\int_0^{2\pi}\int_\mathbb{R}\int_{K_0}\Big|\sum_{n=1}^N(a_n(x_0) \partial_kz_n(\eta_n(x,x_0),y)(S_0^{\prime\prime}(x_0)+i)(x-x_0)\\
&+\varepsilon A_1^n(0,x,y;x_0))\Big|^2e^{-\frac{(x-x_0)^2}{2\varepsilon}}dx_0dxdy\\
&\leq\frac{2}{\sqrt{2\pi\varepsilon}}\Big(\int_0^{2\pi}\int_\mathbb{R}\int_{K_0}\Big|\sum_{n=1}^N(a_n(x_0) \partial_kz_n(\eta_n(x,x_0),y)(S_0^{\prime\prime}(x_0)+i)(x-x_0)\Big|^2\\
&\cdot e^{-\frac{(x-x_0)^2}{2\varepsilon}}dx_0dxdy+\varepsilon^2 \int_0^{2\pi}\int_\mathbb{R}\int_{K_0}\Big|\sum_{n=1}^NA_1^n(0,x,y;x_0))\Big|^2e^{-\frac{(x-x_0)^2}{2\varepsilon}}dx_0dxdy\Big)\\
&=I_1+I_2.
\end{align*}
Switching the order of integration and applying the change of variable for fixed $x_0,\quad$ $$\displaystyle\xi=\frac{x-x_0}{\sqrt{2\varepsilon}},\quad dx=\sqrt{2\varepsilon}d\xi$$  together with the fact that $S_0^{\prime\prime}(x_0)$ is real,
\begin{eqnarray*}
I_1&\leq&\frac{1}{\sqrt{2\pi\varepsilon}}\int_0^{2\pi}\int_{K_0}\int_\mathbb{R}\Big|\sum_{n=1}^Na_n(x_0) \partial_kz_n(\eta_n(\xi,x_0),y)(S_0^{\prime\prime}(x_0)+i)\Big|^22\varepsilon\xi^2e^{-\xi^2}\sqrt{2\varepsilon}d\xi dx_0dy\\
&\leq& \frac{2\varepsilon}{\sqrt{\pi}}\max_{k,1\leq n\leq N} \|\partial_kz_n(k,y)\|_{L^\infty_{y}}^2\int_\mathbb{R}\xi^2e^{-\xi^2}d\xi\int_0^{2\pi}\int_{K_0}\Big|\sum_{n=1}^Na_n(x_0)\Big|^2(S_0^{\prime\prime2}(x_0)+1)dx_0 dy\\
&\leq&2\pi\varepsilon\max_{k,1\leq n\leq N} \| \partial_kz_n(k,y)\|_{L^\infty_{y}}^2\int_{K_0}\Big|\sum_{n=1}^Na_n(x_0)\Big|^2(S_0^{\prime\prime2}(x_0)+1)dx_0.
\end{eqnarray*}
Since $N<\infty$ is finite, the right hand side is bounded by $C\varepsilon$ where constant $C$ depends on the initial data.

As for $I_2,$
\begin{equation*}
I_2=\frac{\varepsilon^{\frac{3}{2}}}{\sqrt{2\pi}}\int_0^{2\pi}\int_\mathbb{R}\int_{K_0}\Big|\sum_{n=1}^NA_1^n(0,x,y;x_0))\Big|^2e^{-\frac{(x-x_0)^2}{2\varepsilon}}dx_0dxdy,
\end{equation*}
we use the definition of $A_1$ in (\ref{A_1}), where we use that $(H(k,y)-E(k))^{-1}$ is bounded operator (moreover, it is compact), hence $A_1$ is bounded. Also, the same change of variable as in the case of $I_1$ estimate will produce the additional rate of convergence.

Hence,
$$I_2\leq C\varepsilon^2$$ and may be neglected since its order of convergence is higher than for $I_1.$
\end{proof}
Using the triangle inequality, we thus get the estimate for the initial error:
$$\|\tilde{\Psi}_0-\tilde{\Psi}_0^\varepsilon\|\leq\|\tilde{\Psi}_0-\Psi^*\|+\|\Psi^*-\tilde{\Psi}_0^{\varepsilon}\|\leq C\varepsilon^{1/2}.$$
\section{Evolution Error - proof of Theorem \ref{Th2.2}}
We prove Theorem \ref{Th2.2} in several steps, in one dimensional setting; an extension to multi-dimensions will be given in next section.
Taking advantage of the band structure of the asymptotic construction and the linearity of the Schr\"odinger operator, we rewrite
$$P(\tilde{\Psi}^\varepsilon)=P\Big(\sum_{n=1}^N\tilde{\Psi}^{\varepsilon n}\Big)=\sum_{n=1}^NP(\tilde{\Psi}^{\varepsilon n}),$$
where
$\tilde{\Psi}^{\varepsilon n}$ is defined in (\ref{3.24}).
By the Minkowski inequality,
$$\quad\|P(\tilde{\Psi}^\varepsilon)\|\leq\sum_{n=1}^N\|P(\tilde{\Psi}^{\varepsilon n})\|.$$
Using residual representation of $P(\tilde{\Psi}^{\varepsilon, n} )$ from  (\ref{3.28}) in  section 3,  we have
 $$P(\tilde{\Psi}^{\varepsilon n})=\sum_{j=0}^2I_{jn},$$ where
\begin{equation}\label{5.1}
I_{jn}=\frac{\varepsilon^{j-\frac{1}{2}}}{(2\pi)^{\frac{1}{2}}}\int_{K_0}G_{jn}(t,x;x_0,y)(x-\tilde{x}_n(t;x_0))^{(3-2j)_+}e^{i\Phi_n(t,x;x_0)/\varepsilon}dx_0,
\end{equation}
where
\begin{align}
G_{0n}(t,x;x_0,y)&=\frac{1}{3!}a_n(t;x_0)\partial_x^3F_n(t,x^*)z_n(k_n,y),\\
G_{1n}(t,x;x_0)&=(ia\langle\partial_kz_n,z_n\rangle\dot{M}-\frac{1}{3!}\partial_x^3F_n(t,x^*)A_{1n}(x-\tilde{x}_n)^2).\\
G_{2n}(t,x;x_0,y)&=a_n(t;x_0)M_n^2\partial_k^2z_n(k_n,y)+iLA_{1n}.
\end{align}
Let $^\prime$ denote quantities defined on the ray emanating from $x_0^\prime$ such as $\tilde{x}_n^\prime, c_{jn}^\prime$ and $\Phi_n^\prime$.

{From Lemma \ref{lem3.1} and Lemma \ref{lem3.2} it follows the following bound:
\begin{equation}\label{gjn}
\int_0^{2\pi}|G_{jn}\overline{G_{jn}^\prime}|dy \leq  C_1.
\end{equation}
Here we note that $G_{1n}$ contains a term involving $(x-\tilde{x}_n)^2$ which becomes unbounded when $x$ is far away from the ray $\tilde{x}_n.$ In such case, the Gaussian beam factor $e^{-\delta|x-\tilde{x}_n|^2/\varepsilon}$ needs to be taken into account.}

We  compute the $L^2$ norm of $I_{jn} $ by
 \begin{align*}
\|I_{jn}\|^2& = \int_0^{2\pi} \int_{\mathbb{R}}  I_{jn}(t,x;x_0,y) \cdot \overline {I_{jn}(t,x;x_0',y)} dxdy  \\
& = \int_0^{2\pi}\int_{\mathbb{R}}\int_{K_0}\int_{K_0}J_{jn}(x, y, x_0, x_0')  dx_0dx_0^\prime dxdy,
\end{align*}
where
\begin{equation}\label{5.3}
J_{jn}(x, y, x_0, x_0')= \frac{\varepsilon^{2j-1}}{2\pi} G_{jn}\overline{G_{jn}^\prime}(x-\tilde{x}_n)^{(3-2j)_+}(x-\tilde{x}_n^\prime)^{(3-2j)_+}e^{i\psi_n/\varepsilon}
\end{equation}
with
\begin{equation}
\psi_n(t, x, x_0, x_0^\prime)=\Phi_n(t,x;x_0)-\overline{\Phi_n^\prime(t,x;x_0^\prime)}.
\end{equation}

Let $\rho_j(x,x_0,x_0')\in C^{\infty}$ be a partition of unity such that
\begin{equation}
\rho_2=\begin{cases}1,\quad |x-\tilde{x}_n|\leq \eta\cap|x-\tilde{x}_n^\prime|\leq \eta,\\
0,\quad|x-\tilde{x}_n|\geq 2\eta\cup|x-\tilde{x}_n^\prime|\geq 2\eta,
\end{cases}
\end{equation}
and $\rho_1+\rho_2 = 1$. Moreover, let
$$
J_{jn}^1=\rho_1J_{jn}(x,y, x_0, x_0'),\qquad
J_{jn}^2=\rho_2J_{jn}(x, y, x_0, x_0') ,
$$
so that $J_{jn}(x,y, x_0, x_0')=J_{jn}^1+J_{jn}^2$.

The rest of this section is to establish the
following
\begin{equation}\label{Izz}
\left| \int_0^{2\pi}\int_{\mathbb{R}} \int_{K_0}\int_{K_0} J_{jn}^i dx_0d{x_0'} dxdy\right| \leq C\varepsilon^{3}
\end{equation}
for $i=1,2$. With this estimate  we  have
$
\|I_{jn}\| \leq C\varepsilon^{\frac{3}{2}}, $
leading to the desired estimate.
{Since for $j=2$ we already have the needed convergence rate, the following proof will be concerned with $j=0$ or $j=1$ cases.}
\subsubsection{Estimate of $J_{jn}^1$}
Using that $\Im\psi_n=\Im\Phi_n+\Im\Phi_n'\geq\delta(|x-\tilde{x}_n|^2+|x-\tilde{x}_n'|^2)$ 
and the definition of $\rho_1$,  in $J_{jn}^1$ either $ |x-\tilde{x}_n(t;x_0)|\quad\mbox{or}\quad|x-\tilde{x}_n(t;x_0^\prime)|$ is greater than $2\eta,$ hence
$$
\int_0^{2\pi}|J^1_{jn}| dy\leq C e^{-\frac{\delta}{2\varepsilon}|x-\tilde{x}_n|^2}e^{-\frac{2\eta^2\delta}{\varepsilon}},
$$
we thus obtain an exponential decay
\begin{equation*}
\left| \int_0^{2\pi}\int_{\mathbb{R}} \int_{K_0}\int_{K_0} J_{jn}^1 dx_0d{x_0'} dxdy\right| \leq C \Big(\frac{2\pi\varepsilon}{\delta}\Big)^{\frac{1}{2}}|K_0|^2e^{-\frac{2\eta^2\delta}{\varepsilon}}\leq C\varepsilon^s\quad\forall s.
\end{equation*}
\subsubsection{Estimation of $J_{jn}^2$}
Using the estimate
$$\displaystyle s^pe^{-as^2}\leq\Big(\frac{p}{e}\Big)^{p/2}a^{-p/2}e^{-as^2/2},
$$
with  $ s=|x-\tilde{x}_n|\quad \mbox{or}\quad s=|x-\tilde{x}^\prime_n|, \; p=3,1, \mbox{or}\; 0, \; a=\frac{\delta}{2\varepsilon}$,  we have
\begin{align*}
\int_0^{2\pi}|J^i_{jn}| dy & \leq CC_2 \varepsilon^2e^{-\frac{\delta}{2\varepsilon}(|x-\tilde{x}_n|^2+|x-\tilde{x}^\prime_n|^2)},
\end{align*}
where
$$
C_2\leq\frac{1}{2\pi} \left(\frac{6}{e\delta}\right)^{3/2}.
$$

Next we note that
\begin{equation}\label{sp}
|x-\tilde{x}_n(t;x_0)|^2+|x-\tilde{x}_n(t;x_0^\prime)|^2=2\Big|x-\frac{\tilde{x}_n(t;x_0)+\tilde{x}_n(t;x_0^\prime)}{2}\Big|^2+\frac{1}{2}|\tilde{x}_n(t;x_0)-\tilde{x}_n(t;x_0^\prime)|^2,
\end{equation}
with which we have
\begin{align*}
\int_0^{2\pi} \int_{\mathbb{R}} |J_{jn}^{2}|dxdy  &\leq C \varepsilon^2  \int_{\mathbb{R}}e^{-\frac{\delta}{\varepsilon}x^2} dx e^{-\frac{\delta}{4\varepsilon}|\tilde{x}_n(t;x_0)-\tilde{x}_n(t;x_0^\prime)|^2}.
\end{align*}
Hence,
\begin{equation}\label{4.9}
\left| \int_0^{2\pi}\int_{\mathbb{R}} \int_{K_0}\int_{K_0} J_{jn}^2 dx_0d{x_0'} dxdy\right| \leq C\varepsilon^{\frac{5}{2}}\int_{K_0}\int_{K_0}e^{-\frac{\delta}{4\varepsilon}|\tilde{x}_n(t;x_0)-\tilde{x}_n(t;x_0^\prime)|^2}dx_0dx_0^\prime.
\end{equation}
In order to obtain (\ref{Izz}), we need to recover an extra  $\displaystyle\varepsilon^{\frac{1}{2}}$ from the integral on the right hand side,   which is difficult when
$\displaystyle |\tilde{x}_n(t;x_0)-\tilde{x}_n(t;x_0^\prime)|$ is small.

Following \cite{LRT11}, we split the set  $K_0\times K_0$ into
$$
D_1(t,\theta)=\Big\{(x_0,x_0^\prime) : |\tilde{x}_n(t,x_0)-\tilde{x}_n(t,x_0^\prime)|\geq\theta|x_0-x_0^\prime|\Big\},
$$
which corresponds to the  non-caustic region of the solution, and the set  associated with the caustic region
$$
D_2(t,\theta)=\Big\{(x_0,x_0^\prime) : |\tilde{x}_n(t,x_0)-\tilde{x}_n(t,x_0^\prime)| < \theta|x_0-x_0^\prime|\Big\}.
$$
For the former we have
$$
\int_{D_1} e^{-\frac{\delta}{4\varepsilon}|\tilde{x}_n(t;x_0)-\tilde{x}_n(t;x_0^\prime)|^2}dx_0dx_0^\prime \leq \int_{D_1} e^{-\frac{\delta \theta^2}{4\varepsilon}|x_0-x_0^\prime|^2}dx_0dx_0^\prime.
$$
Letting $\Lambda =\sup_{x_0, x_0'\in K_0} |x_0-x_0'|< \infty$ be the diameter of $K_0$, we continue
to estimate the above $D_1$-integral
\begin{align*}
 \int_{D_1} e^{-\frac{\delta \theta^2}{4\varepsilon}|x_0-x_0^\prime|^2}dx_0dx_0^\prime  &\leq  C \int_0^\Lambda e^{-\frac{ \delta\theta^2}{4\varepsilon} \tau^2} d\tau \leq C \varepsilon^{1/2},
\end{align*}
which concludes the estimate of $J_{jn}^2$ when restricted on $D_1$ in (\ref{4.9}).

To estimate $J^2_{jn}$ restricted on $D_2$, we need the following result on phase estimate.
\begin{lemm}\textbf{(Phase estimate)}
For $(x_0, x_0')\in D_2$, it holds
$$|\partial_x\psi_n(t,x,x_0,x_0^\prime)|\geq C(\theta,\eta)|x_0-x_0^\prime|,$$
where $C(\theta,\eta)$ is independent of $x$ and positive if $\theta$ and $\eta$ are sufficiently small.
\end{lemm}
The proof of this result is due to \cite{LRT11}, where the non-squeezing lemma  is crucial. Since all requirements  for the non-squeezing argument are satisfied by the construction of Gaussian beam solutions in present work, we therefore omit  details of the proof.

To continue, we note that the phase estimate ensures that for $(x_0, x_0')\in D_2$, $x_0\neq x_0^\prime$,  $\partial_x\psi_n(t,x,x_0,x_0^\prime)\neq0.$  Therefore,  in order to estimate $J_{jn}^2 |_{D_2}$ we shall use  the following non-stationary phase lemma.
\begin{lemm}\textbf{(Non-stationary phase lemma)}
Suppose that $u(x,\xi)\in C_0^\infty(\Omega\times Z$ where $\Omega$ and $Z$ are compact sets and $\psi(x;\xi)\in C^\infty(O)$ for some open neighborhood $O$ of $\Omega\times Z.$ If $\partial_x\psi$ never vanishes in $O,$ then for any $K=0,1,\dots,$
$$\Big|\int_{\Omega}u(x;\xi)e^{i\psi(x;\xi)/\varepsilon}dx\Big|\leq C_K\varepsilon^K\sum_{\alpha=1}^K\int_{\Omega}\frac{|\partial_x^{\alpha}u(x;\xi)|}{|\partial_x\psi(x;\xi)|^{2K-\alpha}}e^{-\Im\psi(x;\xi)/\varepsilon}dx,$$
where $C_K$ is a constant independent of $\xi.$
\end{lemm}
Using the non-stationary lemma, we obtain for $(x_0, x_0')\in D_2$,

\begin{align*}
\left| \int_0^{2\pi}  \int_{\mathbb{R}} J^2_{jn}dxdy \right|  &=\frac{C_K\varepsilon^{K+2j-1}}{2\pi} \times \\
&  \quad \int_0^{2\pi}\int_{\mathbb{R}} \sum_{\alpha=1}^K\frac{|\partial_x^{\alpha}[\rho_1G_{jn}\overline{G_{jn}^\prime}(x-\tilde{x}_n)^{3-2j}(x-\tilde{x}_n^\prime)^{3-2j}]|}{|\partial_x\psi_n(t,x;x_0,x_0^\prime)|^{2K-\alpha}}e^{-\Im\psi_n(x;\xi)/\varepsilon} dxdy.
\end{align*}

By Leibniz's rule,
\begin{eqnarray*}
\partial_x^{\alpha}[\rho_1G_{jn}\overline{G_{jn}^\prime}(x-\tilde{x}_n)^{3-2j}(x-\tilde{x}_n^\prime)^{3-2j}]&=&\sum_{\alpha_1+\alpha_2=\alpha}(\partial_x^{\alpha_1}[\rho_1G_{jn}\overline{G_{jn}^\prime}]\\
&+&\partial_x^{\alpha_2}[(x-\tilde{x}_n)^{3-2j}(x-\tilde{x}_n^\prime)^{3-2j}]).
\end{eqnarray*}
{Here we take a detailed look at the term $\Big|\displaystyle\int_0^{2\pi}\partial_x^{\alpha_1}[\rho_1G_{jn}\overline{G_{jn}^\prime}]dy\Big|$, for each case when $j=0,1.$
For $j=0$, we have
\begin{align*}
\Big|\int_0^{2\pi}\partial_x^{\alpha_1}[\rho_1G_{0n}\overline{G_{0n}^\prime}]dy\Big|&=\Big|\int_0^{2\pi}a_n\overline{a_n'}\partial_x^{\alpha_1}(\rho_1\partial_x^3F_nz_n\overline{\partial_x^3F'_nz'_n})dy\Big|\\
&=\Big|\sum_{\alpha_{11}+\alpha_{12}=\alpha_1}a_n\overline{a_n'}\partial_x^{\alpha_{11}}[\rho_1\partial_x^3F_n\overline{\partial_x^3F_n'}]\cdot\int_0^{2\pi}\partial_x^{\alpha_{12}}[z_n\overline{z'}_n]dy\Big|\\
&\leq|a_n|^2|M_n|^{\alpha_{12}}
\cdot\sum_{\alpha_{11}+\alpha_{12}=\alpha_1}\partial_x^{\alpha_{11}}[\rho_1\partial_x^3F_n\overline{\partial_x^3F_n'}]\cdot\int_0^{2\pi}\partial_k^{\alpha_{12}}[z_n\overline{z'}_n]dy\Big|\\
&\leq CZ^2:=C_2.
\end{align*}}
{
For $j=1$,  we notice that $G_{1n}$ does not depend on $y,$
\begin{align*}
\Big|\partial_x^{\alpha_1}[\rho_1G_{1n}\overline{G_{1n}^\prime}]\Big|&\sim\Big| a_n\overline{a_n'}\partial_x^{\alpha_1}(\rho_1\dot{M}_n\langle\partial_kz_n,z_n\rangle\overline{\dot{M}^\prime_n\langle\partial_kz_n',z_n'\rangle})\Big|\\
&=\Big|\sum_{\alpha_{11}+\alpha_{12}=\alpha_1}a_n\overline{a_n'}\partial_x^{\alpha_{11}}[\rho_1\dot{M}_n\overline{\dot{M}_n'}]\cdot\partial_x^{\alpha_{12}}[\langle\partial_k z_n,z_n\rangle\overline{\langle \partial_kz'_n,z_n\rangle}]\Big|\\
&\leq|a_n|^2|M_n|^{\alpha_{12}}\cdot\sum_{\alpha_{11}+\alpha_{12}=\alpha_1}|\partial_x^{\alpha_{11}}[\rho_1\dot{M}_n\overline{\dot{M}_n^\prime}]\partial_k^{\alpha_{12}}[\langle\partial_k z_n,z_n\rangle\overline{\langle \partial_kz'_n,z_n\rangle}]|\\
&\leq CZ^2:=C_2.
\end{align*}}
Here we used the fact that indices $\alpha_{12}$ and $\alpha_{13}$ are not greater than $2$ and $j$ is either $0$ or $1,$ which is consistent with the boundedness requirement in (\ref{1.5+}).

Going further,
\begin{align*}
\partial_x^{\alpha_2}[(x-\tilde{x}_n)^{3-2j}(x-\tilde{x}_n^\prime)^{3-2j}]\leq& C\sum_{\alpha_{21}+\alpha_{22}=\alpha_2}(x-\tilde{x}_n)^{3-2j-\alpha_{21}}\\
&\cdot (x-\tilde{x}_n^\prime)^{3-2j-\alpha_{22}},
\end{align*}
we have
\begin{align*}
&\int_0^{2\pi}  \int_{\mathbb{R}}|\partial_x^{\alpha_1}[\rho_1G_{jn}\overline{G_{jn}^\prime}]\partial_x^{\alpha_2}[(x-\tilde{x}_n)^{3-2j}(x-\tilde{x}_n^\prime)^{3-2j}]|e^{-\Im\psi_n/\varepsilon}dxdy\\
&\leq C\sum_{\alpha_{21}+\alpha_{22}=\alpha_2}\int_{\mathbb{R}}|x-\tilde{x}_n|^{3-2j-\alpha_{21}}|x-\tilde{x}_n^\prime|^{3-2j-\alpha_{22}}e^{-\Im\psi_n/\varepsilon}dx \\
&\leq C\varepsilon^{\frac{-\alpha_2}{2}+3-2j}\int_{\mathbb{R}} e^{-\frac{\delta}{\varepsilon}(|x-\tilde x_n|^2+ |x-\tilde{x}_n^\prime|^2)}dx \\
&\leq C\left( \frac{\pi}{\delta}\right)^{1/2}\varepsilon^{\frac{1-\alpha_2}{2}+3-2j}e^{-\frac{\delta}{2\varepsilon}|\tilde{x}_n-\tilde{x}_n^\prime|^2},
\end{align*}
where (\ref{sp}) has been used.  Hence,
\begin{align*}
\left| \int_0^{2\pi}\int_{\mathbb{R}} \int_{D_2} J_{jn}^2 dx_0d{x_0'} dxdy \right| \leq &
 \int_{D_2}e^{-\frac{\delta}{2\varepsilon}|\tilde{x}-\tilde{x}^\prime|^2}\sum_{\alpha=1}^K\frac{\varepsilon^{\frac{\alpha}{2}+2j-1}}{\inf|\partial_x\psi_n/\sqrt{\varepsilon}|^{2K-\alpha}}\\
\cdot\sum_{\alpha_{1}+\alpha_{2}=\alpha} &C \varepsilon^{\frac{1-\alpha_2}{2}+3-2j}dx_0dx_0^\prime \\
\leq &C\varepsilon^{\frac{5}{2}}\int_{D_2} e^{-\frac{\delta}{2\varepsilon}|\tilde{x}_n-\tilde{x}_n^\prime|^2} \sum_{\alpha=1}^K\frac{1}{\inf|\partial_x\psi_n/\sqrt{\varepsilon}|^{2K-\alpha}}
dx_0dx_0^\prime.
\end{align*}
The last estimate together with (\ref{4.9}) yields:
\begin{eqnarray*}
\left| \int J_{jn}^{2}1_{D_2} \right|  &\leq&C\varepsilon^{\frac{5}{2}}\int_{D_2} e^{-\frac{\delta}{2\varepsilon}|\tilde{x}_n-\tilde{x}_n^\prime|^2}\min\Big[1,\sum_{\alpha=1}^K\frac{1}{\inf|\partial_x\psi_n/\sqrt{\varepsilon}|^{2K-\alpha}}\Big]dx_0dx_0^\prime\\
&\leq&C\varepsilon^{\frac{5}{2}}\int_{D_2} e^{-\frac{\delta}{2\varepsilon}|\tilde{x}_n-\tilde{x}_n^\prime|^2}\sum_{\alpha=1}^K\ \min\Big[1,\frac{1}{\inf|\partial_x\psi_n/\sqrt{\varepsilon}|^{2K-\alpha}}\Big]dx_0dx_0^\prime\\
&\leq&C \varepsilon^{\frac{5}{2}}\int_{K_0}\int_{K_0}e^{-\frac{\delta}{2\varepsilon}|\tilde{x}_n-\tilde{x}_n^\prime|^2} \sum_{\alpha=1}^K \frac{1}{1+\inf|\partial_x\psi_n/\sqrt{\varepsilon}|^{2K-\alpha}}dx_0dx_0^\prime\\
&\leq&C \varepsilon^{\frac{5}{2}}\int_{K_0}\int_{K_0}\sum_{\alpha=1}^K\frac{1}{1+(C(\theta,\eta)|x_0-x_0^\prime|/\sqrt{\varepsilon})^{2K-\alpha}}dx_0dx_0^\prime.
\end{eqnarray*}
Taking $K=2$ and changing variable $\displaystyle\xi=\frac{x_0-x_0^\prime}{\sqrt{\varepsilon}}$,  we compute
\begin{align*}
\left| \int J_{jn}^{2}1_{D_2} \right|
&\leq  C  \varepsilon^{\frac{5}{2}} \int_{K_0 \times K_0} \frac{1}{1 + \left(|x_0-x_0'|/\sqrt{\varepsilon} \right)^K} dx_0dx_0'\\
& \leq C \varepsilon^3  \int_0^\infty \frac{1}{1+\xi^K}d\xi =  \frac{\pi}{2} C \varepsilon^3,
\end{align*}
which gives (\ref{Izz}) when restricted to the caustic region.

Putting all together we complete the proof  of (\ref{Izz}), hence Theorem \ref{Th2.2}.

\section{Extensions }
The extension of the one-dimensional results to multidimensional case is straightforward.  We still have the two-scale formulation,
\begin{equation}
i\varepsilon\frac{\partial\tilde{\Psi}}{\partial t}=-\frac{1}{2}(\varepsilon\nabla_x+\nabla_y)^2\tilde{\Psi}+V(\frac{x}{\varepsilon})\tilde{\Psi}+V_e(x)\tilde{\Psi},\quad x\in{\mathbb{R}^d,}
\end{equation}
\begin{equation}\Psi(0,x,y)=g(x,y)e^{iS_0(x)/{\varepsilon}},\quad x\in K_0\subset\mathbb{R}^d,\quad y\in[0,2\pi]^d.
\end{equation}
The Gaussian beam construction of the phase will have the following form:
\begin{equation}
\Phi(t,x;x_0)=S(t;x_0)+p(t;x_0)\cdot(x-\tilde{x}(t;x_0))+\frac{1}{2}(x-\tilde{x}(t;x_0))^\top \cdot M(x-\tilde{x}(t;x_0)).
\end{equation}

Following the procedure of the Gaussian beam construction in section $3$, we only check possible different formulations in the multidimensional setting. For instance, equation (\ref{3.4}) will take a form:
$$c_j=\Delta_xA_{j-2}+iLA_{j-1}+GA_j,\quad j=2,3,\dots,l+2,$$
where $L$ reads
$$
L=\partial_t+(-i\nabla_y+\nabla_x\Phi)\cdot \nabla_x + \frac{1}{2} \Delta_x \Phi.
$$
The evolution equations for the Gaussian beam phase components:
\begin{equation}
\begin{cases} \dot{\tilde{x}}=\nabla_kE(p),\quad \tilde{x}|_{t=0}=x_0, \\
\dot{p}=-\nabla_xV_e(\tilde{x}),\quad p|_{t=0}=\nabla_xS(x_0),\\
\dot{S}=p\cdot\nabla_k E(p)-E(p)-V_e(\tilde{x}),\quad S|_{t=0}=S_0(x_0),\\
\dot{M}=-M \nabla_k^2E(p) M- \nabla_x^2V_e(\tilde{x}),\quad M|_{t=0}=\nabla_x^2 S_0(x_0)+iI.\\
\end{cases}
\end{equation}
An equation for the amplitude can be derived from (\ref{c1}), however because of the matrix $M,$ it has more sophisticated form than in $1$- dimensional case:
\begin{equation}\label{amp}
\dot{a}=a(\langle(\nabla_kz\cdot(\nabla_xV_e(\tilde{x})+M\nabla_kE(p)),z\rangle-\langle(-i\nabla_y+p)\cdot M\nabla_kz,z\rangle-\frac{1}{2}Tr(M)).
\end{equation}
One can easily verify that the amplitude equation for $d=1$ follows from (\ref{amp}).

The superposition formula (\ref{3.24}) for the approximate solution is:
\begin{equation}
\tilde{\Psi}^{\varepsilon}(t,x,y)=\frac{1}{(2\pi\varepsilon)^{d/2}}\int_{K_0}\tilde{\Psi}_{GB}^\varepsilon(t,x,y,x_0)dx_0.
\end{equation}
The technique for estimating the initial error can be carried out in multi-dimensional setting, without any further difficulty.

As for the evolution error, some clarification of the notation needs to be done. For example, the main representations (\ref{5.1})-(\ref{5.3}),
can be reformulated as follows:
\begin{equation}
I_{jn}=\frac{\varepsilon^{j-\frac{d}{2}}}{(2\pi)^{\frac{d}{2}}}\int_{K_0}\sum_{|\beta|=(3-2j)_+}G_{jn\beta}(t, x, y;x_0)
(x-\tilde{x}_n(t;x_0))^{\beta}e^{i\Phi_n(t,x;x_0)/\varepsilon}dx_0,
\end{equation}
{where
\begin{align}
G_{0n\beta}(t,x;x_0,y)&=\frac{1}{\beta!}a_n(t;x_0)\partial_x^\beta F_n(t,x^*)z_n(k,y),\quad |\beta|=3,\\
G_{1n\beta}(t,x;x_0)&=(ia\langle\partial_kz_n,z_n\rangle\dot{M}_n(t;x_0)-\sum_{|\beta|=3} \frac{1}{\beta!}\partial_x^\beta F_n(t,x^*)A_{1n}(x-\tilde{x})^{(\beta-1)_+}),\\
G_{2n}(t,x;x_0,y)&=a_n(t;x_0)(Tr(M_n))^2\Delta_kz_n(k_n,y)+iLA_{1n}.
\end{align}}
Finally,
\begin{align*}
J_{jn}(x, y, x_0, x_0')&=\frac{\varepsilon^{2j-d}}{(2\pi)^d}\sum_{|\beta|=(3-2j)_+}(G_{jn\beta}(t, x, y;x_0))(x-\tilde{x}_n(t;x_0))^{\beta}\\
&\times \sum_{|\beta|=(3-2j)_+}\overline{(G_{jn\beta}(t, x, y;x_0'))}(x-\tilde{x}_n(t;x_0'))^{\beta}e^{i\psi_n/\varepsilon}.
\end{align*}

The rest of the ingredients of the proof remain unchanged, {except  when using the non-stationary phase method $K$ need to be taken as  $d+1$}.

Another possible extension of this result is to apply our technique to higher order Gaussian beam superpositions, using the Gaussian beam construction in \cite{DGR06}.

{Our results valid for finite number of bands can be used in practice by approximating a given high frequency initial data by finite number of bands within certain accuracy.
An open question is to deal with infinite number of bands,  which is left in a future work. }

\section*{Acknowledgments}
This research was  partially supported by the National Science Foundation under Grant DMS 09-07963 and DMS 13-12636.

\bigskip
\bibliographystyle{abbrv}

\end{document}